\documentclass[reqno]{amsart}

\usepackage{hyperref}
\usepackage{amsmath,amsthm,amssymb}

\newtheorem{thm}{Theorem}

\newtheorem{lem}[thm]{Lemma}

\newtheorem{clm}[thm]{Claim}

\theoremstyle{definition}
\newtheorem{dfn}[thm]{Definition}
\newtheorem{rem}[thm]{Remark}
\newtheorem{prob}[thm]{Problem}

\numberwithin{thm}{section}
\numberwithin{equation}{section}

\title[Collapsing to Alexandrov spaces]{Collapsing to Alexandrov spaces with isolated mild singularities}
\author[T. Fujioka]{Tadashi Fujioka}
\address{Department of Mathematics, Kyoto University, Kitashirakawa Oiwakecho, Sakyo-ku, Kyoto 606-8502, Japan}
\email{\href{mailto:tfujioka@math.kyoto-u.ac.jp}{tfujioka@math.kyoto-u.ac.jp}, \href{mailto:tfujioka210@gmail.com}{tfujioka210@gmail.com}}
\date{\today}
\thanks{Supported by JSPS KAKENHI Grant Number 15H05739}

\begin{document}

\begin{abstract}
Let $M_j$ be a sequence of Riemannian manifolds with sectional curvature bound below collapsing to a compact Alexandrov space $X$ of dimension $k$.
Suppose that all but finitely many points of $X$ are $(k,\delta)$-strained and that the space of directions at each exceptional point contains $k+1$ directions making obtuse angles with each other.
We prove that $M_j$ admits a structure of locally trivial fibration over $X$ for sufficiently large $j$.
The same is true for collapsing sequences of Alexandrov spaces such that the infimum of the volume of the spaces of directions is sufficiently large relative to $\delta$.
\end{abstract}

\maketitle

\section{Introduction}\label{sec:int}
The goal of the collapsing theory of Riemannian manifolds is to describe the topology of collapsing manifolds with a lower sectional curvature bound in terms of the geometry of limit Alexandrov spaces.
In general, singular fibers arise over singular strata of a limit space.
However, if a limit space satisfies some regularity condition, one can expect that collapsing manifolds admit fibration structures over the limit space.

Suppose a sequence $M_j$ of $n$-dimensional Riemannian manifolds with sectional curvature $\ge\kappa$ converges to a compact Alexandrov space $X$ of dimension $k<n$ in the Gromov-Hausdorff topology.
Yamaguchi's fibration theorem \cite{Y:col} asserts that if $X$ is a Riemannian manifold,  or more generally every point of $X$ is $(k,\delta)$-strained, then $M_j$ admits a structure of locally trivial fibration over $X$ for sufficiently large $j$ (cf.\ \cite{Y:conv}, \cite{RX}, \cite{X}, \cite{XY}, \cite{F:fibr}).
On the other hand, Perelman \cite{Per:col} showed that if $X$ has no singular strata called extremal subsets, then $M_j$ admits a kind of Serre fibration structure over $X$ (cf.\ \cite{F:good}).
While Perelman's conclusion is weaker than Yamaguchi's one, the assumption is optimal.
There is a large gap between these two theorems.
For example, in case $X$ is a polyhedral surface, Yamaguchi's assumption means that the sum of angles at each vertex (which is not greater than $2\pi$) is almost $2\pi$, whereas Perelman's assumption means that it is greater than $\pi$.

In this paper we extend Yamaguchi's fibration theorem to the case where $X$ contains finitely many exceptional points which are less regular than $(k,\delta)$-strained points.
We say that a point $p\in X$ is \textit{weakly $k$-strained} if the space of directions $\Sigma_p$ contains $k+1$ directions making obtuse angles with each other.
This type of regularity was introduced and studied by Perelman \cite{Per:alex}, \cite{Per:mor}, \cite{Per:dc}.
In particular there is a distance coordinate around a weakly $k$-strained point.

The following is the main result of this paper.
We will omit the lower curvature bound $\kappa$ unless otherwise stated.
A positive number $\delta$ is less than some constant depending only on $n$ and $\kappa$.

\begin{thm}\label{thm:main}
Let $X$ be a $k$-dimensional compact Alexandrov space such that all points are $(k,\delta)$-strained except for finitely many weakly $k$-strained points, where $k<n$.
Let $M$ be an $n$-dimensional Riemannian manifold sufficiently close to $X$ in the Gromov-Hausdorff distance.
Then $M$ admits a structure of locally trivial fibration over $X$.
\end{thm}

In case $X$ is a polyhedral surface, our assumption means that the sum of angles at each vertex is greater than $3\pi/2$.
More generally, if $X$ is a $2$-dimensional Alexandrov space without boundary, all but finitely many points are $(2,\delta)$-strained (\cite{BGP}).
Furthermore $X$ is a topological manifold and any space of directions is homeomorphic to a circle (\cite{Per:alex}).
In this case our assumption is equivalent to the condition that every space of directions has length $>3\pi/2$.
Note that the absence of boundary is necessary to obtain a fibration structure since it is an extremal subset (\cite{PP:ext}).

In a previous paper \cite{F:fibr}, the author generalized Yamaguchi's fibration theorem to collapsing sequences of Alexandrov spaces with a lower bound on the volume of the spaces of directions.
The same generalization works in this case as well.

\begin{thm}\label{thm:main'}
Theorem \ref{thm:main} holds true if $M$ is an Alexandrov space such that the infimum of the volume of the spaces of directions is greater than some constant depending only on $n$, $\kappa$, and $\delta$.
\end{thm}

The proof of the main theorems is an application of the method developed in \cite{F:fibr}, where the author gave an alternative proof of Yamaguchi's fibration theorem.
Around a $(k,\delta)$-strained point, one can define a local fibration by lifting the distance map from a strainer.
The key observation in \cite{F:fibr} is that any two local fibrations defined by different strainers have almost the same differential quotients.
Hence if one glue them by taking an average of them, the differential quotient remains almost unchanged.
Since the resulting map has almost the same regularity as a strainer map, it becomes a locally trivial fibration.

In this paper the limit space is allowed to have finitely many exceptional points which are weakly $k$-strained.
As in the case of strained points, one can define a local fibration around a weakly strained point.
We show that this local fibration can be glued to the above strainer-like map, with little change in the differential quotient.
In particular the resulting map has almost the same regularity as the original map and is a locally trivial fibration.

The following are natural problems in this direction.

\begin{prob}
Extend Yamaguchi's fibration theorem to the case where every point of $X$ is weakly $k$-strained.
\end{prob}

\begin{prob}
Extend Yamaguchi's fibration theorem to the case where $\dim X=2$ and $X$ has no proper extremal subsets; in other words every space of direction is a circle of length $>\pi$.
\end{prob}

This paper is organized as follows.
Sec.\ \ref{sec:nac} contains notation and conventions.
Sec.\ \ref{sec:pre} provides preliminaries on strained points and weakly strained points.
Sec.\ \ref{sec:prf} is devoted to a direct proof of Theorem \ref{thm:main}, which is much simpler than the proof of Theorem \ref{thm:main'}.
The generalization to Theorem \ref{thm:main'} needs a modified version of the generalized fibration theorem proved in \cite{F:fibr}, which will be discussed in Sec.\ \ref{sec:rev}.
\section{Notation and conventions}\label{sec:nac}

Positive integers $n$ and $k$ denote the dimension of a collapsing space and a limit space, respectively.
The lower curvature bound $\kappa$ will be omitted.
A lower bound $\ell$ for the lengths of strainers is uniformly bounded above when $\kappa<0$.

We denote by $c$ and $C$ various small and large positive constants, respectively, which depend only on $n$ and $\kappa$ unless otherwise stated.
The dependence on additional parameters will be indicated explicitly, like $c(\varepsilon)$.
A positive number $\delta$ is assumed to be less than some constant depending only on $n$ and $\kappa$.
We denote by $\varkappa(\delta)$ various positive functions such that $\varkappa(\delta)\to0$ as $\delta\to0$ which depends only on $n$ and $\kappa$.
If there exists another positive number $\varepsilon$, the choices of $\delta$ and $\varkappa(\delta)$ may depend additionally on $\varepsilon$ (especially in Sec.\ \ref{sec:rev}).

Let $M$ be a collapsing space and $X$ a limit space.
We fix a Gromov-Hausdorff approximation between $M$ and $X$.
For $p\in X$ we denote by $\hat p\in M$ the image of $p$ under this approximation and call it a lift of $p$.
In particular one can lift distance functions.
If a map $\varphi:X\to\mathbb R^l$ is built up out of distance functions, we denote by $\hat\varphi:M\to\mathbb R^l$ its natural lift.

\section{Preliminaries}\label{sec:pre}

We refer to \cite{BGP} and \cite{BBI} for the basics of Alexandrov spaces.
See also \cite[\S3]{F:fibr} for the basic properties of strained points.

Let $X$ be an Alexandrov space.
We first define strained points introduced by Burago-Gromov-Perelman \cite{BGP}.
We denote by $\tilde\angle$ the comparison angle.

\begin{dfn}\label{dfn:str}
A point $p\in X$ is said to be \textit{$(k,\delta)$-strained} if there exists $\{(a_i,b_i)\}_{i=1}^k$ in $X$ such that
\[\tilde\angle a_ipb_i>\pi-\delta,\quad\tilde\angle a_ipa_j,\tilde\angle a_ipb_j,\tilde\angle b_ipb_j>\pi/2-\delta\]
for any $i\neq j$.
The collection $\{(a_i,b_i)\}_{i=1}^k$ is called a \textit{$(k,\delta)$-strainer} at $p$.
The value $\min_i\{|a_ip|,|b_ip|\}$ is called the \textit{length} of this strainer.
\end{dfn}

Let $\{(a_i,b_i)\}_{i=1}^k$ be a $(k,\delta)$-strainer at $p\in X$, where $k=\dim X$.
Then the distance map $\varphi=(|a_1\cdot|,\dots,|a_k\cdot|):X\to\mathbb R^k$ is a $\varkappa(\delta)$-almost isometric open embedding near $p$ (\cite{BGP}).
Let $\hat\varphi$ denote a natural lift of $\varphi$ to a collapsing space $M$.
One can define $\varphi^{-1}\circ\hat\varphi$ on a neighborhood of a lift of $p$.
Moreover it is a locally trivial fibration over a neighborhood of $p$ (\cite{Per:alex}).

Next we define a weakly strained point.
This type of regularity was introduced by Perelman \cite{Per:alex}, \cite{Per:mor}, \cite{Per:dc}.

\begin{dfn}
A point $p\in X$ is said to be \textit{weakly $k$-strained} if there exists $\{a_i\}_{i=1}^{k+1}$ in $X$ such that
\[\tilde\angle a_ipa_j>\pi/2\]
for any $i\neq j$.
The collection $\{a_i\}_{i=1}^{k+1}$ is called a \textit{weak $k$-strainer} at $p$.
\end{dfn}

Let $\{a_i\}_{i=1}^{k+1}$ be a weak $k$-strainer at $p\in X$, where $k=\dim X$.
Then the distance map $\varphi=(|a_1\cdot|,\dots,|a_k\cdot|):X\to\mathbb R^k$ is a bi-Lipschitz open embedding near $p$ (\cite{Per:alex}, \cite{Per:mor}, \cite{Per:dc}).
Let $\hat\varphi$ denote a natural lift of $\varphi$ to a collapsing space $M$.
One can define $\varphi^{-1}\circ\hat\varphi$ on a neighborhood of a lift of $p$.
Moreover it is a locally trivial fibration over a neighborhood of $p$ (\cite{Per:alex}, \cite{Per:mor}).

If $M$ is a Riemannian manifold, the proof of the local triviality of the above maps is much easier, as we will see in the next section.

\section{Proof of the main theorem}\label{sec:prf}

In this section we prove Theorem \ref{thm:main}.
We first construct a map from $M$ to $X$ and show the key property of this map (Claim \ref{clm:key}).
These construction and property actually hold for a general Alexandrov space $M$, which is not necessarily a Riemannian manifold.
We then suppose that $M$ is a Riemannian manifold and prove Theorem \ref{thm:main}, using the key property.
This key property also allows us to prove Theorem \ref{thm:main'}, but the detail is deferred to the next section since it is much more complicated.

Let us first recall the key theorem \cite[4.1]{F:fibr}.
We reformulate it as follows to be suitable for our setting.
The hat symbol $\hat\ $ will indicate a natural lift (see Sec.\ \ref{sec:nac}).

\begin{thm}\label{thm:str}
Let $X$ be an Alexandrov space of dimension $k<n$.
Let $A\subset X$ be a domain such that every point has a $(k,\delta)$-strainer with length $>\ell$.
Suppose an $n$-dimensional Alexandrov space $M$ is sufficiently close to $X$.
Then there exist a neighborhood $U\subset M$ of a lift of $A$ and a map $f:U\to X$ close to the Gromov-Hausdorff approximation satisfying the following property:
\begin{itemize}
\item[($\ast$)]
Let $(a,b)$ be a $(1,\delta)$-strainer at $p\in A$ with $|ap|=|bp|=\ell\delta$.
Then we have
\[\bigl|(|af(x)|-|af(y)|)-(|\hat ax|-|\hat ay|)\bigr|<\varkappa(\delta)|xy|\]
for any $x,y\in B(\hat p,\ell\delta^2)\cap U$.
\end{itemize}
\end{thm}

\begin{rem}\label{rem:fibr}
The property ($\ast$) together with the generalized fibration theorem \cite[5.2]{F:fibr} immediately implies that $f:f^{-1}(A)\to A$ is a locally trivial fibration (under the assumption on the volume of the spaces of directions of $M$ as in Theorem \ref{thm:main'}).
If $M$ is a Riemannian manifold, the proof is much easier as we will see later.
\end{rem}

\begin{rem}\label{rem:str}
The upper bound on the length of a strainer in the property ($\ast$) is only used for the inductive proof and is not really necessary.
Indeed let $(\alpha,\beta)$ be a $(1,\delta)$-strainer at $p\in A$ with length $>\ell\delta$.
Let $a$ and $b$ be points on shortest paths $p\alpha$ and $p\beta$ at distance $\ell\delta$ from $p$, respectively.
Clearly $(a,b)$ is a $(1,\delta)$-strainer satisfying the assumption of ($\ast$).
By the property of a strainer, we have
\begin{gather*}
\bigl|(|\alpha f(x)|-|\alpha f(y)|)-(|af(x)|-|af(y)|)\bigr|<\varkappa(\delta)|xy|,\\
\bigl|(|\hat\alpha f(x)|-|\hat\alpha f(y)|)-(|\hat ax|-|\hat ay|)\bigr|<\varkappa(\delta)|xy|
\end{gather*}
for any $x,y\in B(\hat p,\ell\delta^2)$ (\cite[3.2]{F:fibr}).
Combining these with the property ($\ast$) yields
\[\bigl|(|\alpha f(x)|-|\alpha f(y)|)-(|\hat\alpha x|-|\hat\alpha y|)\bigr|<\varkappa(\delta)|xy|.\]
\end{rem}

Now we construct a map from $M$ to $X$, where $X$ is a $k$-dimensional Alexandrov space satisfying the assumption of Theorems \ref{thm:main} and \ref{thm:main'} and $M$ is an $n$-dimensional Alexandrov space sufficiently close to $X$.
Note that neither the assumption that $M$ is a Riemannian manifold nor the assumption on the volume of the spaces of directions of $M$ will be used until the end of the proof of Claim \ref{clm:key}.

Since the following argument is local on $X$, we may assume that there is only one weakly $k$-strained point $p\in X$.
Fix small $r>0$ and let $\ell>0$ be such that every point of $X\setminus B(p,r)$ has a $(k,\delta)$-strained with length $>\ell$.
These constants will be adjusted later.
By Theorem \ref{thm:str} there exist a neighborhood $U$ of $M\setminus B(\hat p,r)$ and a map $f:U\to X$ satisfying the property $(\ast)$.
By Remark \ref{rem:fibr} the restriction of $f$ to $f^{-1}(X\setminus B(p,r))$ is a locally trivial fibration (under the above-mentioned assumptions on $M$).

We will construct a local fibration around $p$ and glue it to $f$.
Let $R>0$ be so small that the blow up $R^{-1}B(p,R)$ is sufficiently close to the unit ball in the tangent cone $T_p$.
Since $p$ is weakly $k$-strained, there exist $a_1,\dots,a_k,w\in\partial B(p,R)$ such that
\[\tilde\angle a_ipa_j>\pi/2+\varepsilon,\quad\tilde\angle a_ipw>\pi/2+\varepsilon\]
for any $i\neq j$ and some $\varepsilon>0$.
As mentioned in Sec.\ \ref{sec:pre}, the distance map from $a_i$ gives a coordinate around $p$ and its lift defines a fibration around $\hat p$.
Here we modify it as follows.
For small $\omega>0$, which will be determined later, take a maximal $\omega$-discrete set $\{a_{i\alpha}\}_{\alpha=1}^{N_i}$ from the $\varepsilon R/10$-neighborhood of each $a_i$.
Then we have
\[\tilde\angle a_{i\alpha}pa_{j\beta}>\pi/2+\varepsilon/2,\quad\tilde\angle a_{i\alpha}pw>\pi/2+\varepsilon/2\]
for any $i\neq j$, $\alpha$, and $\beta$ if $R$ is small enough.
Furthermore the Bishop-Gromov inequality implies
\begin{equation}\label{eq:vol1}
N_i\ge\frac{c(v,\varepsilon,R)}{\omega^k},
\end{equation}
where $v>0$ is a lower bound for the volume of $B(p,1)$.
Define $\varphi=(\varphi_1,\dots,\varphi_k):X\to\mathbb R^k$ by
\[\varphi_i:=\frac1{N_i}\sum_{i=1}^{N_i}\varphi_{i\alpha},\quad\varphi_{i\alpha}:=|a_{i\alpha}\cdot|.\]
A map of this form, called an admissible map, was introduced and studied in \cite{Per:mor}.
We assume that the angle conditions above hold on $B(p,10r)$ for $r\ll R$.
It was shown in \cite{Per:mor} that $\varphi$ is a $c(\varepsilon)$-bi-Lipschitz open embedding on $B(p,10r)$.
Hence if $\hat\varphi:M\to\mathbb R^k$ denotes a natural lift of $\varphi$, then $\varphi^{-1}\circ\hat\varphi$ is defined on $B(\hat p,5r)$.
Furthermore its restriction to $\hat\varphi^{-1}(\varphi(B(p,3r)))$ is a locally trivial fibration over $B(p,3r)$ (\cite{Per:mor}).
A direct proof in the Riemannian case will be given later.

We now glue $\varphi^{-1}\circ\hat\varphi$ to $f$ as follows:

\begin{equation*}
g:=
\begin{cases}
\hfil\varphi^{-1}\circ\hat\varphi & \text{on }B(\hat p,r)\\
\varphi^{-1}((1-\chi_{\hat p})\varphi\circ f+\chi_{\hat p}\hat\varphi) & \text{on }B(\hat p,2r)\setminus B(\hat p,r)\\
\hfil f & \text{otherwise}
\end{cases}
\end{equation*}
where $\chi_{\hat p}:=\chi(|\hat p\cdot|/r)$ and $\chi:[0,\infty)\to[0,1]$ is a smooth function such that $\chi\equiv 1$ on $[0,1]$ and $\chi\equiv0$ on $[2,\infty)$.

We show that this gluing procedure hardly changes the differential quotient.

\begin{clm}\label{clm:key}
For any sufficiently close $x,y$ in a neighborhood of $\bar B(\hat p,2r)\setminus B(\hat p,r)$, we have
\[\bigl|(\varphi\circ g(x)-\varphi\circ g(y))-(\hat\varphi(x)-\hat\varphi(y))\bigr|<\varkappa(\delta)|xy|.\]
\end{clm}

\begin{proof}
Fix $q\in\bar B(p,2r)\setminus B(p,r)$.
Let $C_q$ be the cut locus of $q$, that is, the set of points beyond which shortest paths from $q$ cannot be extended.
Recall that $C_q$ is a Borel set of measure zero (\cite[3.1]{OS}).
Let $U(C_q,\cdot)$ denote metric neighborhoods of $C_q$.
For any $\nu>0$, there exists $\omega_q>0$ such that the $2\omega_q$-neighborhood $U(C_q,2\omega_q)$ has measure at most $\nu$.

Fix $1\le i\le k$.
It suffices to show the claim for each $i$-th coordinate.
Put $A_i:=\{a_{i\alpha}\}_{\alpha=1}^{N_i}$ and $A_i(q):=A_i\cap U(C_q,\omega_q)$.
Let $N_i(q)$ denote the cardinality of $A_i(q)$.
If $\omega\le\omega_q$, the Bishop-Gromov inequality implies
\begin{equation}\label{eq:vol2}
N_i(q)\le\frac{C(v)\nu}{\omega^k},
\end{equation}
where $v$ is a lower bound for the volume of $B(p,1)$ as before.
Together with the inequality \eqref{eq:vol1} this shows that $A_i\setminus A_i(q)$ occupies the vast majority of $A_i$ provided $\nu$ is small enough.

Since a shortest path does not branch in an Alexandrov space, any point of $X\setminus C_q$ can be joined to $q$ by a unique shortest path.
Since $q$ is $(k,\delta)$-strained, the space of directions at $q$ is $\varkappa(\delta)$-close to the unit $(k-1)$-sphere in the Gromov-Hausdorff distance (\cite[3.3]{F:fibr}).
In particular any direction at $q$ has an almost opposite direction.
By the compactness of $X\setminus U(C_q,\omega_q)$, there exists $\ell_q>0$ such that for any $a\in X\setminus U(C_q,\omega_q)$ there exists $b\in X$ such that $(a,b)$ is a $(1,\varkappa(\delta))$-strainer at $q$ with length $>\ell_q$.

Let $r_q<\ell_q\delta$ and take a finite cover of $\bar B(p,2r)\setminus B(p,r)$ by such $B(q_\beta,r_{q_\beta})$.
To simplify the notation we will write $r_\beta=r_{q_\beta}$, $A_{i\beta}=A_i(q_\beta)$, and so on.
We may assume $\ell\le\min_\beta\ell_\beta$ and $\omega\le\min_\beta\omega_\beta$.

Let $x,y$ be as in the assumption.
We may assume $|xy|<\ell\delta^2$ and $x\in B(\hat q_\beta,r_\beta)$ for some $\beta$.
In particular for any $a_{i\alpha}\in A_i\setminus A_{i\beta}$ there exists $b_{i\alpha}\in X$ such that $(a_{i\alpha},b_{i\alpha})$ is a $(1,\varkappa(\delta))$-strainer at $q_\beta$ with length $>\ell$.
Hence by Theorem \ref{thm:str} and Remark \ref{rem:str} we have
\[\bigl|(\varphi_{i\alpha}(f(x))-\varphi_{i\alpha}(f(y)))-(\hat\varphi_{i\alpha}(x)-\hat\varphi_{i\alpha}(y))|\bigr|<\varkappa(\delta)|xy|\]
for any $1\le\alpha\le N_i$ such that $a_{i\alpha}\in A_i\setminus A_{i\beta}$.
Setting $\psi_{i\alpha}:=(1-\chi_{\hat p})\varphi_{i\alpha}\circ f+\chi_{\hat p}\hat\varphi_{i\alpha}$, we have
\begin{align*}
&(\psi_{i\alpha}(x)-\psi_{i\alpha}(y))-(\hat\varphi_{i\alpha}(x)-\hat\varphi_{i\alpha}(y))\\
&=(1-\chi_{\hat p}(x))((\varphi_{i\alpha}\circ f(x)-\varphi_{i\alpha}\circ f(y))-(\hat\varphi_{i\alpha}(x)-\hat\varphi_{i\alpha}(y)))\\
&\qquad\qquad\qquad\qquad\qquad-(\chi_{\hat p}(x)-\chi_{\hat p}(y))(\varphi_{i\alpha}\circ f(y)-\hat\varphi_{i\alpha}(y)).
\end{align*}
Note that the second term is bounded above by an arbitrarily small multiple of $|xy|$ if $M$ is sufficiently close to $X$ since $\chi_{\hat p}$ is Lipschitz continuous and $f$ is close to the Gromov-Hausdorff approximation.
Therefore we obtain
\begin{equation}\label{eq:clm1}
\bigl|(\psi_{i\alpha}(x)-\psi_{i\alpha}(y))-(\hat\varphi_{i\alpha}(x)-\hat\varphi_{i\alpha}(y))\bigr|<\varkappa(\delta)|xy|
\end{equation}
for any $1\le \alpha\le N_i$ such that $a_{i\alpha}\in A_i\setminus A_{i\beta}$.

On the other hand, for any $1\le \alpha\le N_i$ such that $a_{i\alpha}\in A_{i\beta}$ we have
\begin{equation}\label{eq:clm2}
\bigl|(\psi_{i\alpha}(x)-\psi_{i\alpha}(y))-(\hat\varphi_{i\alpha}(x)-\hat\varphi_{i\alpha}(y))\bigr|<C|xy|
\end{equation}
by the Lipschitz continuity.

As mentioned before, the inequalities \eqref{eq:vol1} and \eqref{eq:vol2} imply $N_{i\beta}/N_i\approx 1$ provided $\nu$ is small enough.
Taking the average of the inequalities \eqref{eq:clm1} and \eqref{eq:clm2}, we obtain
\[\bigl|(\varphi_i\circ g(x)-\varphi_i\circ g(y))-(\hat\varphi_i(x)-\hat\varphi_i(y))\bigr|<\varkappa(\delta)|xy|,\]
as required.
\end{proof}

\begin{proof}[Proof of Theorem \ref{thm:main}]
Now suppose $M$ is a Riemannian manifold.
Let us prove that $\varphi\circ g$ restricted to $g^{-1}(B(p,3r))$ is a locally trivial fibration over $\varphi(B(p,3r))$.
Since $M$ is a manifold, it suffices to show that $\varphi\circ g$ is a topological submersion, i.e.\ locally a bundle map (see \cite[Remark]{CK} or \cite[6.14]{Si}).
Fix $z\in \bar B(\hat p,2r)\setminus B(\hat p,r)$ (the other case is similar).
By choosing $\hat a_{i\alpha}$ outside of the cut locus of $z$, we may assume that $\hat\varphi_{i\alpha}$ and $\hat\varphi_i$ are smooth near $z$.
Then, if $M$ is sufficiently close to $X$, angle comparison shows
\[\langle\nabla_z\hat\varphi_i,\nabla_z\hat\varphi_j\rangle<-c(\varepsilon),\quad\langle\nabla_z\hat\varphi_i,\hat w'_z\rangle>c(\varepsilon)\]
for any $i\neq j$, where $\hat w'_z$ denotes the direction of a shortest path from $z$ to $\hat w$.
In particular this implies that $\nabla_z\hat\varphi_i$ are linearly independent (it is easy to see by the Gram-Schmidt process, for example).
Let $\xi_{k+1},\dots,\xi_n\in T_z$ be an orthogonal basis of the orthogonal complement of the subspace spanned by $\nabla_z\hat\varphi_i$.
Let $b_i\in M$ be points in the directions $\xi_i$ outside the cut locus of $z$ and set $f_i:=|b_i\cdot|$ for $k+1\le i\le n$.
We also put $f_i:=\hat\varphi_i$ for $1\le i\le k$ and $F:=(f_1,\dots,f_n)$.

The inverse function theorem implies that $F$ is a local diffeomorphism near $z$.
Moreover for any $\xi\in\Sigma_z$ there exists $1\le i\le n$ such that $|\langle\nabla_zf_i,\xi\rangle|>c(\varepsilon)$, where $c(\varepsilon)$ is much smaller than the previous one (otherwise the orthogonal projections of $\nabla_zf_i$ to the subspace orthogonal to $\xi$ are linearly independent, which contradicts the dimension).
Since $F$ is smooth, this yields
\[|F(x)-F(y)|>c(\varepsilon)|xy|\]
for any $x$, $y$ close to $z$.
Together with Claim \ref{clm:key}, this implies
\[|\tilde F(x)-\tilde F(y)|>c(\varepsilon)|xy|,\]
where $\tilde F:=(\varphi\circ g,f_{k+1},\dots,f_n)$ (we may assume $\delta\ll\varepsilon$ so that $\varkappa(\delta)\ll c(\varepsilon)$).
By invariance of domain,  $\tilde F$ is a bi-Lipschitz homeomorphism near $z$ to an open subset of $\mathbb R^n$.
Therefore $\varphi\circ g$ is a topological submersion near $z$, as desired.
A similar argument using Theorem \ref{thm:str} also shows that $g$ is a locally trivial fibration outside $g^{-1}(B(p,3r))$.
This completes the proof.
\end{proof}

To prove Theorem \ref{thm:main'}, we need the generalized fibration theorem proved in \cite{F:fibr}.
Roughly speaking, this theorem says that a map $F$ satisfying the inequality
\[\bigl|(F(x)-F(y))-(G(x)-G(y))\bigr|<\delta|xy|\]
for some noncritical distance map $G$ is a locally trivial fibration.
We call such $F$ a generalized noncritical map.
However, we cannot apply this theorem directly to the map $\varphi\circ g$ in Claim \ref{clm:key} because the map $G$ used in the definition of a generalized noncritical map does not include admissible maps such as $\hat\varphi$.
We will explain how to modify it in the next section.

\section{Generalized fibration theorem revisited}\label{sec:rev}

In this section we finish the proof of Theorem \ref{thm:main'}.
As noted at the end of the previous section, we need to modify the generalized fibration theorem \cite[5.2]{F:fibr} to be suitable for our application.
We will modify the definition of a generalized noncritical map \cite[5.1]{F:fibr} to include admissible maps and show that the generalized fibration theorem still holds for the modified definition.
This section is rather technical and the reader is assumed to be familiar with the proof of the fibration theorem in \cite{Per:alex} (cf.\ \cite{Per:mor}) and its generalization in \cite{F:fibr}.
Here $\delta$ and $\varkappa(\delta)$ may depend additionally on $\varepsilon$ and are assumed to be much less than $\varepsilon$ and $c(\varepsilon)$.

Before describing the modification, we recall the notion of a DER function introduced in \cite{Per:mor}, which is an abstraction of the derivative of an admissible function.
Let $\Sigma$ be a space of curvature $\ge1$.
We call $f:\Sigma\to\mathbb R$ a \textit{DER function} if it has the following form:
\[f=\sum_\alpha-a_\alpha\cos|A_\alpha\cdot|,\]
where $\{A_\alpha\}_\alpha$ is a finite collection of compact subsets of $\Sigma$ and $a_\alpha\ge0$, $\sum_\alpha a_\alpha\le1$.
Note that the differential of $f$ at $p\in\Sigma$
\[f'_p=\sum_\alpha-a_\alpha\sin|A_\alpha p|\cos|(A_\alpha)'_p\cdot|\]
is again a DER function on $\Sigma_p$, where $(A_\alpha)'_p$ denotes the set of directions of shortest paths from $p$ to $A_\alpha$.

For two DER functions $f=\sum_\alpha-a_\alpha\cos|A_\alpha\cdot|$ and $g=\sum_\beta-b_\beta\cos|B_\beta\cdot|$ on $\Sigma$, we define their scalar product by
\[\langle f,g\rangle:=\sum_{\alpha,\beta}a_\alpha b_\beta\cos|A_\alpha B_\beta|.\]
The following properties of the scalar product were proved in \cite[\S2]{Per:mor} (cf.\ \cite[\S2]{Per:dc}).
Both follow from triangle comparison.
\begin{itemize}
\item For any DER functions $f,g$ on $\Sigma$ and $p\in\Sigma$, we have
\begin{equation}\label{eq:der1}
\langle f',g'\rangle_p\le\langle f,g\rangle-f(p)g(p).
\end{equation}
\item For any DER function $f$ on $\Sigma$, there exists $\hat f=-\hat a\cos|\hat A\cdot|$ such that
\begin{equation}\label{eq:der2}
\langle\hat f,g\rangle\le\langle f,g\rangle
\end{equation}
for any DER function $g$ on $\Sigma$, where $\hat A\in\Sigma$ and $0\le\hat a\le 1$.
\end{itemize}

The above properties allow us to prove the following DER versions of \cite[2.2--2.4]{Per:alex} (or in other words the $(\varepsilon,\delta)$-version of \cite[2.3]{Per:mor}, \cite[2.2]{Per:dc}).
Suppose $\dim\Sigma=n-1$.

\begin{lem}\label{lem:der1}
There exist no DER functions $f_1,\dots,f_{n+2}$ on $\Sigma$ such that $\langle f_1,f_i\rangle<-\varepsilon$ for any $i\ge 3$, $\langle f_i,f_j\rangle<\delta$ for any $i\neq j$, and $\max f_2>\varepsilon$, where $\delta\ll\varepsilon$.
\end{lem}

\begin{lem}\label{lem:der2}
\begin{enumerate}
\item Let $f_1,\dots, f_{k+1}$ ($k\le n$) be DER functions on $\Sigma$ such that
\[\langle f_1,f_i\rangle<-\varepsilon,\quad\langle f_i,f_j\rangle<\delta\]
for any $i\neq j\ge2$, where $\delta\ll\varepsilon$.
Then there exists $\xi\in\Sigma$ such that
\[f_1(\xi)>c(\varepsilon),\quad f_2(\xi)<-c(\varepsilon),\quad f_i(\xi)=0\]
for any $i\ge 3$.
\item (1) holds true if we replace the assumption $\langle f_1,f_2\rangle<-\varepsilon$ by $\langle f_1,f_2\rangle<\delta$ and $\max f_2>\varepsilon$, and the conclusion $f_1(\xi)>c(\varepsilon)$ by $f_1(\xi)>-\varkappa(\delta)$.
\item Under the assumptions of (1) there exists $\xi\in\Sigma$ such that $f_2(\xi)>c(\varepsilon)$ and $f_i(\xi)=0$ for any $i\ge 3$.
\end{enumerate}
\end{lem}

\begin{rem}\label{rem:der}
\begin{enumerate}
\item Under the assumption of Lemma \ref{lem:der2}(1) there exists $\eta\in\Sigma$ such that $f_i(\eta)>\varepsilon$ for any $i\ge 2$.
This immediately follows from \eqref{eq:der2}.
\item The assumption $\max f_2>\varepsilon$ in Lemma \ref{lem:der1} and Lemma \ref{lem:der2}(2) is necessary since the scalar product is positively homogeneous.
Note that this assumption is weaker than $\langle f_1,f_2\rangle<-\varepsilon$ as we have seen in (1).
\item Lemma \ref{lem:der2}(2) is only used in the proof of \cite[3.3]{Per:alex} (\cite[5.7]{F:fibr}), which is not actually necessary for the proof of the fibration theorem.
\end{enumerate}
\end{rem}

The proofs are just minor modifications of the original ones.
Here are outlines.

\begin{proof}[Proof of Lemma \ref{lem:der1}]
Let us prove it by induction on $n$.
The base case is obvious.
The induction step is as follows.
Let $\hat f_2=-\hat a_2\cos|\hat A_2\cdot|$ be as in \eqref{eq:der2}.
Note that the assumption $\max f_2>\varepsilon$ together with \eqref{eq:der2} implies $\hat a_2>\varepsilon$.
Hence for any $i\neq 2$ we have
\[-\hat a_2 f_i(\hat A_2)=\langle\hat f_2,f_i\rangle\le\langle f_2,f_i\rangle<\delta.\]
Since $\hat a_2>\varepsilon$ this gives $f_i(\hat A_2)>-\varkappa(\delta)$.
Therefore the inequality \eqref{eq:der1} implies
\begin{align*}
\langle f_1',f_i'\rangle_{\hat A_2}&\le\langle f_1,f_i\rangle-f_1(\hat A_2)f_i(\hat A_2)\\
&<-\varepsilon+\varkappa(\delta)<-c(\varepsilon).
\end{align*}
Similarly we have $\langle f_i',f_j'\rangle_{\hat A_2}<\varkappa(\delta)$ for any $i\neq j\ge 3$.
This completes the induction step (see also Remark \ref{rem:der}(2)).
\end{proof}

\begin{proof}[Proof of Lemma \ref{lem:der2}]
(1)
We prove it by induction on $n$.
The base case is obvious.
For the induction step, set
\[X:=\left\{x\in\Sigma\mid f_1(x)\ge c(\varepsilon),\ f_i(x)\ge0\ (i\ge3)\right\}.\]
First we show $X\neq\emptyset$.
Let $\hat f_2=-\hat a_2\cos|\hat A_2\cdot|$ be as in \eqref{eq:der2}.
An argument similar to the above proof shows
\[f_1(\hat A_2)>\varepsilon,\quad f_i(\hat A_2)>-\varkappa(\delta)\]
for any $i\ge 3$ and thus $\langle f_1',f_i'\rangle_{\hat A_2}<-c(\varepsilon)$.
In particular by Remark \ref{rem:der}(1) there exists a direction at $\hat A_2$ that increases the value of $f_i$ with velocity at least $c(\varepsilon)$ for any $i\ge 3$.
An easy argument based on this observation shows $X\neq\emptyset$.

For any $x\in X$ the inequality \eqref{eq:der1} implies
\[\langle f_1',f_i'\rangle_x<-\varepsilon,\quad\langle f_i',f_j'\rangle_x<\delta\]
for any $i\neq j\ge 3$.
Using the induction hypothesis and the elementary lemmas \cite[2.1.2, 2.1.3]{Per:alex}, one can find $\zeta\in\Sigma$ such that $f_1(\zeta)\ge c(\varepsilon)$ and $f_i(\zeta)=0$ for any $i\ge 3$.
Let $\xi$ be a minimum point of $f_2$ among such $\zeta$.
A similar argument shows $f_2(\xi)<-c(\varepsilon)$, which completes the proof.

(2)
We prove it by reverse induction on $k$.
Let $\hat f_2=-\hat a_2\cos|\hat A_2\cdot|$ be as in \eqref{eq:der2}.
An argument similar to (1) shows that there exists $\eta\in\Sigma$, $\varkappa(\delta)$-close to $\hat A_2$, such that $f_1(\eta)>-\varkappa(\delta)$ and $f_i(\eta)\ge0$ for any $i\ge3$.
Using (1) and \cite[2.1.2, 2.1.3]{Per:alex}, one can find $\zeta\in\Sigma$ such that
\[f_1(\zeta)>-\varkappa(\delta)+c(\varepsilon)|\zeta\eta|,\quad f_i(\zeta)=0\]
for any $i\ge 3$.

If $f_2(\zeta)<-c(\varepsilon)$ we are done.
Suppose the contrary: $f_2(\zeta)>-\delta$, where $\delta=c(\varepsilon)$ will be determined later.
Since $\hat f_2(\hat A_2)=-\hat a_2<-\varepsilon$, we have $f_2(\hat A_2)<-\varepsilon$ by \eqref{eq:der2}.
Hence we have $|\zeta\hat A_2|>\varepsilon/2$.
Since $\eta$ is $\varkappa(\delta)$-close to $\hat A_2$, this implies $|\zeta\eta|>\varepsilon/3$.
By the definition of $\zeta$, we have $f_1(\zeta)>c(\varepsilon)$.
Therefore the collection $f_1$, \dots, $f_{k+1}$, $f_{k+2}=-\cos|\zeta\cdot|$ satisfies the assumption of (2), which gives the induction step.
The base case $k=n$ follows from Lemma \ref{lem:der1}.

(3)
This easily follows from (1) and \cite[2.1.2]{Per:alex}.
See the original proof \cite[2.4]{Per:alex}.
\end{proof}

Now we explain how to modify the definition of a generalized noncritical map introduced in \cite{F:fibr}.
Here is the original definition \cite[5.1]{F:fibr}.

\begin{dfn}\label{dfn:nc}
Let $U$ be an open subset of an Alexandrov space $M$ and $p\in U$.
A map $f=(f_1,\dots,f_k):U\to\mathbb R^k$ is said to be \textit{$(\varepsilon,\delta)$-noncritical in the generalized sense} at $p$ if it satisfies the following conditions:
\begin{enumerate}
\item There exists a function $g_i:U\to\mathbb R$ such that
\[\bigl|(f_i(x)-f_i(y))-(g_i(x)-g_i(y))\bigr|<\delta|xy|\]
for any $x,y\in U$ and
\[g_i=\inf_\gamma g_{i\gamma},\quad g_{i\gamma}=\varphi_{i\gamma}(|A_{i\gamma}\cdot|)+\sum_{l=1}^{i-1}\varphi_{i\gamma}^l(f_l(\cdot))+c_{i\gamma},\]
where $c_{i\gamma}\in\mathbb R$, $A_{i\gamma}$ are compact subsets of $M$, $\varphi_{i\gamma}$ and $\varphi_{i\gamma}^l$ have right and left derivatives, $\varphi_{i\gamma}^l$ are $\varepsilon^{-1}$-Lipschitz functions, and $\varphi_{i\gamma}$ are increasing functions such that $\varphi_{i\gamma}(0)=0$ and $\varepsilon|x-y|\le|\varphi_{i\gamma}(x)-\varphi_{i\gamma}(y)|\le\varepsilon^{-1}|x-y|$.
\item The sets of indices $\Gamma_i(p):=\{\gamma\mid g_i(p)=g_{i\gamma}(p)\}$ satisfy $\#\Gamma_i(p)<\varepsilon^{-1}$.
Furthermore there exists $\rho=\rho(p)>0$ such that for all $i$
\[g_i(x)<g_{i\gamma}(x)-\rho\]
for any $x\in B(p,\rho)$ and $\gamma\notin\Gamma_i(p)$.
\item $\tilde\angle A_{i\alpha}pA_{j\beta}>\pi/2-\delta$ for any $i\neq j$, $\alpha\in\Gamma_i(p)$, $\beta\in\Gamma_j(p)$.
\item There exists $w=w(p)\in M$ such that $\tilde\angle A_{i\gamma}pw>\pi/2+\varepsilon$ for any $i$ and $\gamma\in\Gamma_i(p)$.
\end{enumerate}
\end{dfn}

Let us modify the above definition as follows:

\begin{itemize}
\item Divide the index set $\{1,\dots,k\}$ into two subsets $I$ and $J$.
For each $i\in I$ we assume $g_i$ is defined by the same formula as in (1).
For each $i\in J$ we change the definition of $g_i$ to
\[g_i=\frac1{N_i}\sum_{\gamma=1}^{N_i}|A_{i\gamma}\cdot|\]
where $A_{i\gamma}$ are compact subsets of $M$.
In this case (2) can be ignored.
\item Define DER functions
\[d_pg_i:=-\cos|A_i'(p)\cdot|,\quad A_i'(p):=\bigcup_{\gamma\in\Gamma_i(p)}(A_{i\gamma})'_p\]
for $i\in I$ and
\[d_pg_i:=-\frac1{N_i}\sum_{\gamma=1}^{N_i}\cos|(A_{i\gamma})'_p\cdot|\]
for $i\in J$.
\item Instead of (3) we assume that $\langle d_pg_i,d_pg_j\rangle<\delta$ for any $i\neq j$.
\item Instead of (4) we assume that there exists $\xi\in\Sigma_p$ such that $d_pg_i(\xi)>\varepsilon$ for any $i$.
\end{itemize}

Note that the change from comparison angle to real angle in (3) and (4) does not matter since the definition is local and the angle is lower semicontinuous.

With these modifications one can repeat the proof of the generalized fibration theorem in \cite[\S5]{F:fibr}.
In fact all we have to do is replace the conditions for angles in the original proof by the corresponding ones for DER functions, as we have done in Lemmas \ref{lem:der1} and \ref{lem:der2}.
The details are left to the reader.

Therefore the generalized fibration theorem \cite[5.2]{F:fibr} holds for the above modified definition and can be applied to the map $\varphi\circ g$ in Claim \ref{clm:key}.
This completes the proof of Theorem \ref{thm:main'}.


\begin{thebibliography}{99}

\bibitem{BBI}
D. Burago, Y. Burago, and S. Ivanov, A course in metric geometry, Grad. Stud. Math., 33, Amer. Math. Soc., Providence, RI, 2001.

\bibitem{BGP}
Yu. Burago, M. Gromov, and G. Perel'man, A.D. Alexandrov spaces with curvature bounded below, Uspekhi Mat. Nauk 47 (1992), no. 2(284), 3--51, 222; translation in Russian Math. Surveys 47 (1992), no. 2, 1--58.

\bibitem{CK}
J. Cheeger and J. Kister, Counting topological manifolds, Topology 9 (1970), no. 2, 149--152.

\bibitem{F:fibr}
T. Fujioka, A fibration theorem for collapsing sequences of Alexandrov spaces, to appear in J. Topol. Anal., \href{https://doi.org/10.1142/S179352532150028X}{DOI:10.1142/S179352532150028X}, 2021.

\bibitem{F:good}
T. Fujioka, Application of good coverings to collapsing Alexandrov spaces, to appear in Pacific J. Math., \href{https://arxiv.org/abs/2010.02520}{arXiv:2010.02520}, 2020.

\bibitem{OS}
Y. Otsu and T. Shioya, The Riemannian structure of Alexandrov spaces, J. Differential Geom. 39 (1994), no. 3, 629--658.

\bibitem{Per:alex}
G. Perelman, Alexandrov's spaces with curvatures bounded from below II, preprint, 1991.

\bibitem{Per:mor}
G. Ya. Perel'man, Elements of Morse theory on Aleksandrov spaces, Algebra i Analiz 5 (1993), no. 1, 232--241; translation in St. Petersburg Math. J. 5 (1994), no. 1, 205--213.

\bibitem{Per:dc}
G. Perelman, DC structure on Alexandrov space, preprint, 1994.

\bibitem{Per:col}
G. Perelman, Collapsing with no proper extremal subsets, Comparison geometry, 149--154, Math. Sci. Res. Inst. Publ., 30, Cambridge Univ. Press, Cambridge, 1997.

\bibitem{PP:ext}
G. Ya. Perel'man and A. M. Petrunin, Extremal subsets in Aleksandrov spaces and the generalized Liberman theorem, Algebra i Analiz 5 (1993), no. 1, 242--256; translation in St. Petersburg Math. J. 5 (1994), no. 1, 215--227.

\bibitem{RX}
X. Rong and S. Xu, Stability of $e^\epsilon$-Lipschitz and co-Lipschitz maps in Gromov-Hausdorff topology, Adv. Math. 231 (2012), 774--797.

\bibitem{Si}
L. C. Siebenmann, Deformation of homeomorphisms on stratified sets, Comment. Math. Helv. 47 (1972), 123--163.

\bibitem{X}
S. Xu, Homotopy lifting property of an $e^\epsilon$-Lipschitz and co-Lipschitz map, preprint, \href{https://arxiv.org/abs/1211.5919}{arXiv:1211.5919}, 2013.

\bibitem{XY}
S. Xu and X. Yao, Margulis lemma and Hurewicz fibration theorem on Alexandrov spaces, to appear in Commun. Contemp. Math., \href{https://doi.org/10.1142/S0219199721500486}{DOI:10.1142/S0219199721500486}, 2021.

\bibitem{Y:col}
T. Yamaguchi, Collapsing and pinching under a lower curvature bound, Ann. of Math. 133 (1991), no. 2, 317--357.

\bibitem{Y:conv}
T. Yamaguchi, A convergence theorem in the geometry of Alexandrov spaces, Actes de la table ronde de g\'eom\'etrie diff\'erentielle, 601--642, S\'emin. Congr., 1, Soc. Math. France, Paris, 1996.

\end{thebibliography}
\end{document}